%% file: Row_red_250123_arXiv.tex
\documentclass[11pt,a4paper]{article}
\usepackage{amsmath}
\usepackage{amssymb}
\usepackage{amsthm}
\usepackage{amsfonts}
\usepackage{latexsym}
\usepackage{cite}
\usepackage{array}
\usepackage{mathtools}
\usepackage{enumerate}

\usepackage{a4wide}

\usepackage{url}

\usepackage{hyperref}
\usepackage{type1cm}        
\usepackage{makeidx}         
\usepackage{graphicx}        

\usepackage{multicol}        
\usepackage[bottom]{footmisc}

\theoremstyle{plain}
  \newtheorem{theorem}{Theorem}
  
  \newtheorem{lemma}{Lemma}
  \newtheorem{corollary}{Corollary}
  
\theoremstyle{definition}
  \newtheorem{definition}{Definition}
  
\theoremstyle{remark}
  \newtheorem{remark}{Remark}

\allowdisplaybreaks

\usepackage{bm}
\usepackage{siunitx}

\usepackage{amsmath}

\newcounter{algorithm}

\input{macros}

\usepackage[skip=10pt]{subcaption}
\usepackage[noabbrev,capitalize]{cleveref}

\newcommand{\akCT}{\widetilde{C}}
\newcommand{\AKabs}[1]{\left\vert#1\right\vert}

\title{Reduced digital nets}

\author{V. Anupindi, P. Kritzer}
\date{}

\begin{document}

\maketitle              

\begin{abstract}

In the recent papers ``The fast reduced QMC matrix-vector product'' (J. Comput. Appl. Math. 440, 115642, 2024) and ``Column reduced digital nets'' (submitted), it was proposed to use QMC rules based on reduced digital nets which provide a speed-up in the computation of QMC vector-matrix products that may occur in practical applications. In this paper, we provide upper bounds on the quality parameter of row reduced and column-row reduced digital nets, which are helpful for the error analysis of using reduced point sets as integration nodes in a QMC rule. We also give remarks on further aspects and comparisons of row reduced, column reduced, and column-row reduced digital net.

\end{abstract}

\section{Introduction}\label{AKsec:intro}

\subsection{The problem setting}
In various fields like statistics, finance, and uncertainty quantification, one needs to numerically compute integrals of the form
\begin{equation}\label{AKeq:integral}
 \int_D f(\bsx^\top A) \rd \mu(\bsx), 
\end{equation}
where $A$ is a real $s\times \tau$ matrix,  $D\subseteq\RR^s$, and $\mu$ is a suitable measure, by quasi-Monte Carlo (QMC) rules,  
\begin{equation}\label{AKeq:QMC_approx}
Q_N (f):=\frac{1}{N}\, \sum_{k=0}^{N-1} f(\bsx_k^\top A).
\end{equation}
Here, the $\bsx_k =(x_k^{(1)},\ldots,x_k^{(s)})^\top$, $k\in\{0,1,\ldots,N-1\}$, are column vectors corresponding to the points used in the QMC rule. For the sake of simplicity, we assume that $D=[0,1]^s$, and choose $\mu$ as the Lebesgue measure in the following.

We are interested in the cases where the vector-matrix product $\bsx_k^\top A$ is a significant factor in the computation of \eqref{AKeq:QMC_approx}. In special circumstances, one can modify $A$ such that it allows a fast vector-matrix multiplication, see \cite{DN97,GKNSS11}. However, in this paper, we do not modify the matrix $A$, but we focus on obtaining special quadrature points $\bsx_0, \dots, \bsx_{N-1}$ that speed-up the computation of the vector-matrix product $\bsx_k^\top A$ for $k\in\{0,1,\ldots,N-1\} $. We define the $N \times s$-matrix
\begin{equation}\label{AKeq:def_mat_X}
    X  = \begin{pmatrix}
    \bsx_0^\top \\
    \vdots \\
    \bsx_{N-1}^\top \\
\end{pmatrix}, 
\end{equation}
i.e., the $k$-th row is the $k$-th point $\bsx_k$ of the quadrature rule. Computing the vector-matrix products $\bsx_k^\top A$ 
for $k\in\{0,1,\ldots,N-1\}$
is equivalent to computing the matrix-matrix product $XA$, which requires $\calO (N \, s \, \tau)$ operations.
In \cite{DKLS15}, the authors showed that by using QMC point sets that are derived from (polynomial) lattice rules or Korobov's $p$-sets, one can obtain a speed-up in the matrix-matrix multiplication $XA$ by reordering $X$ to be of circulant structure and using the fast Fourier transformation (FFT). In this case, the cost of evaluating $Q_N(f)$ in \eqref{AKeq:QMC_approx} can be reduced to $\calO( \tau \, N \, \log N)$ operations provided that $\log N \ll s$.

In this paper, we take a different route and focus on quadrature points obtained from \emph{reduced digital nets.} Digital nets, as we know, are efficient and among the most commonly used QMC node sets. To obtain \emph{reduced} digital nets, we set a certain number of rows, columns, or both rows and columns of the generating matrices $C_j^{(m)}$ to zero. The main motivation for doing this is to obtain a certain repetitive structure in the components of the points $\bsx_k$ which can be used to obtain fast computation of the products $\bsx_k ^\top A$. 

The idea of using reduced digital nets was proposed in \cite{DEHKL23}, where the authors looked at \emph{row reduced} digital nets, that is, digital nets where one sets certain rows of the generating matrices to zero. In \cite{AK24}, we extended the study to column reduced digital nets.

While studying such reduced digital nets, we have four main aspects to consider. The first two pertain to the theoretical aspects, particularly the quality parameter (the so-called $t$-value) of the reduced net and the error estimate for using reduced nets in quadrature rules. The practical aspects include a fast algorithm for computing the matrix-matrix product using points from a reduced net and implementing the algorithm to obtain some numerical results on the computation time. 

In \cite{DEHKL23}, the authors provided a fast matrix-matrix product algorithm using row reduced digital nets and also provided error estimates assuming that the $t$-value of the row reduced digital net is known. However, an estimate for the $t$-value of the row reduced digital net was still missing in \cite{DEHKL23}. In \cite{AK24}, the authors of the present paper focused on column reduced digital nets, where we gave an estimate for the $t$-value, an error estimate using this bound, and also provided an algorithm for matrix-matrix multiplication using column reduced digital nets. Furthermore, with numerical experiments, we compared the computation time of row reduced with column reduced digital nets and the standard matrix-matrix multiplication.

In this paper, we will revisit the row reduced digital nets as proposed in \cite{DEHKL23}. We estimate the quality parameter of these nets and provide an error bound using our result in \cref{AKlem:t_value_red_proj} below. For an algorithm to compute the matrix-matrix product using row reduced digital nets, we refer to \cite[Algorithm 4]{DEHKL23}. We also comment on a combination of column and row reduced digital nets. 

In the following, we use $\N$ to denote the set of natural numbers starting from $1$ and $\N_0$ to denote the set of whole numbers, which are all the non-negative integers, i.e., $\N_0 = \{0,1,2,\dots \}$. 

\subsection{Digital nets and sequences}\label{AKsec:digital}
In this section, we recall the definitions and establish notation for digital nets and sequences. 

Given a prime $b$, let $\F_b$ be a finite field with $b$ elements, where the elements of $\F_b$ can be identified with the set $\{0,1,\dots,b-1\}$. An \textit{elementary interval} in base $b$ and dimension $s$ is a half-open interval of the form $\prod_{j=1}^{s} [a_jb^{-d_j}, (a_j + 1)b^{-d_j} )$, where the $a_j,d_j$ are non-negative integers with $0 \leq a_j < b^{d_j}$ for $1 \leq j \leq s$.

\begin{definition} 
For a given dimension $s \geq 1$ and non-negative integers $t,m$ with $0 \leq t \leq m$, a \emph{$(t,m,s)$-net} in base $b$ is a point set $\mathcal{P} \subset [0,1)^s$ consisting of $b^m$ points such that any elementary interval in base $b$ with volume $b^{t-m}$ contains exactly $b^t$ points of $\mathcal{P}$.

A sequence $(\bsx_{0},\bsx_1,\dots)$ of points in $[0,1)^s$ is called a \emph{$(t,s)$-sequence} in base $b$ if, for all integers $m\ge t$ and $k\geq 0$, the point set consisting of the points $\bsx_{kb^m},\dots,\bsx_{kb^m + b^m - 1}$ forms a $(t,m,s)$-net in base $b$. 
\end{definition}

For a $(t,m,s)$-net or a $(t,s)$-sequence, the parameter $t$, also known as the \textit{quality parameter} of a net or sequence, is crucial since it is an indicator of how uniformly distributed the points are. In particular, a low $t$-value is desirable since it implies better equidistribution properties. 

A $(t,m,s)$-net is called \textit{strict}, if it does not fulfill the requirements of a $(t-1,m,s)$-net (for $t\ge 1$), and analogously for $(t,s)$-sequences. In general, any $(t,m,s)$-net is also a $(t+1,m,s)$-net for $t < m$.

To construct concrete examples of good nets in arbitrary dimension, one frequently uses the \emph{digital method} introduced by Niederreiter in \cite{N92a}. 

\begin{definition}\label{AKdef:rho_m}
    A \emph{digital $(t,m,s)$-net over $\F_b$} is a $(t,m,s)$-net $\mathcal{P} = \{\bsx_0,\dots,\bsx_{b^{m}-1} \}$ where the points are constructed as follows. Let $C_1^{(m)},\dots,C_s^{(m)}$ in $\F_b^{m \times m}$ be matrices over $\F_b$. To generate the $k$-th point, $0\le k\le b^m-1$, in $\mathcal{P}$, we use the $b$-adic expansion $k = \sum_{i=0}^{m-1} k_i b^i$ with digits $k_i \in \{0,\dots,b-1\}$ which we denote by $\overrightarrow{k} = (k_0,\dots,k_{m-1})^{\top}$. The $j$-th coordinate $x_{k,j}$ of $\bsx_k = (x_{k,1},\dots,x_{k,s})$ is obtained by computing
    \[
    \overrightarrow{x}_{k,j} := C_j^{(m)} \,\overrightarrow{k},
    \]
    over $\F_b$ and then setting 
    \[
   x_{k,j}:=  \overrightarrow{x}_{k,j} \cdot (b^{-1},b^{-2},\ldots,b^{-m}).
    \]
    Similarly, a \emph{digital} $(t,s)$-sequence $\mathcal{S}$ over $\F_b$ 
    is generated by infinite matrices $C_1,\dots,C_s$, where 
    \begin{equation}\label{AKeq:def_Ci}
    C_j = (c_{i,r}^{(j)})_{i,r \in \N} \in \F_b^{\N \times \N}. 
    \end{equation}
    To generate the $k$-th point, $k\ge 0$, in $\mathcal{S}$, we use the $b$-adic expansion $k = \sum_{i=0}^{\infty} k_i b^i$ with digits $k_i \in \{0,\dots,b-1\}$ which we denote by $\overrightarrow{k} = (k_0,k_1,\dots)^{\top}$. The 
    $j$-th coordinate $x_{k,j}$ of $\bsx_k = (x_{k,1},\dots,x_{k,s})$ is obtained by computing
    \[
    \overrightarrow{x}_{k,j} :=  C_j^{(m)} \,\overrightarrow{k},
    \]
    over $\F_b$ and then setting 
    \[
   x_{k,j}:=  \overrightarrow{x}_{k,j} \cdot (b^{-1},b^{-2},\ldots).
    \]
\end{definition}

The quality parameter of a  digital $(t,m,s)$-net or a $(t,s)$-sequence solely depends on its generating matrices, in particular, it is linked to certain linear independence properties of the rows of the generating matrices.

\begin{definition}
    For any integers $1 \leq j \leq s$ and $m \geq 1$, let 
$C_{1}^{(m)}, C_2^{(m)},\ldots, C_s^{(m)}$ be $m \times m$ matrices over $\F_b$. Let $\bsc_{i}^{(j)}$ denote the $i$-th row of the $j$-th matrix.
Then the \emph{linear independence parameter} $\rho_m 
(C_{1}^{(m)}, C_2^{(m)},\ldots, C_s^{(m)})$ is defined as the largest integer $d$ such that for any choice of $d_1,\dots,d_s \in \N_0$, with $d_1 +\dots+d_s = d$, we have that the system of row vectors
\[
    \{ \bsc_{i}^{(j)} : 1 \leq i \leq d_j, 1 \leq j \leq s  \}
\]
is linearly independent over $\F_b$.
\end{definition}
It is known (see, e.g., \cite{N92,DP10}) that the generating matrices $C_{1}^{(m)}, C_2^{(m)},\ldots, C_s^{(m)}$ of a digital $(t,m,s)$-net over $\F_b$ satisfy 
\begin{equation} \label{AKeq:rho_t_relation}
    \rho_m (C_{1}^{(m)}, C_2^{(m)},\ldots, C_s^{(m)}) \ge m-t,
\end{equation}
where we have equality if the net is a strict $(t,m,s)$-net. Similarly, for the generating matrices $C_1,\ldots,C_s$ of a digital $(t,s)$-sequence over $\F_b$ we must have
$\rho_m(C_1^{(m)},\dots,C_s^{(m)}) \ge m-t$ for all $m\ge \max\{t,1\}$, where 
$C_j^{(m)}$ denotes the left upper $m\times m$ submatrix of $C_j$ for $j\in \{1,\ldots,s\}$.
Hence, for digital nets and sequences, their quality can be assessed by checking linear independence conditions on the rows of the generating matrices. 

Note that from any digital $(t,s)$-sequence over $\F_b$ with generating matrices $C_1,\ldots,C_s$, we can, for $m\ge t$, derive a digital $(t,m,s)$-net over $\F_b$, simply by considering the point set generated by the left upper $m\times m$ submatrices $C_1^{(m)},\ldots,C_s^{(m)}$ of $C_1,\ldots,C_s$. This is equivalent to considering the first $b^m$ points of the $(t,s)$-sequence. 

For more details on digital nets and sequences, we refer to \cite{DP10,NW15}.

\section{Row reduced digital nets}

In \cite{DEHKL23}, the authors introduced row reduced digital nets to obtain a speed-up in the matrix-matrix multiplication described in the introduction. They provide an algorithm which exploits a certain structure in the generation of row reduced digital nets and they also give error bounds. However, obtaining estimates for the quality parameter of row reduced digital nets is left as an open problem in \cite{DEHKL23}. 

In this section we provide estimates for the quality parameter of row reduced digital nets and we also use these estimates for error analysis.

Let us consider a digital $(t,m,s)$-net generated by the matrices $C_1^{(m)},\dots,C_s^{(m)}$.
 
Let $w_1,w_2,\ldots,w_s\in\N_0$ with $0=w_1 \le \cdots  \le w_s$. We call these numbers the \textit{reduction indices} for the 
generating matrices $C_j^{(m)} = (c_{i,r}^{(j)})$, $i,r\in \{1,2,\ldots,m\}$. We derive the corresponding row reduced matrices $\akCT_1^{(m)},\ldots,\akCT_s^{(m)}$, with $\akCT_j^{(m)}=(\widetilde{c}_{i,r}^{(j)})$, $i,r\in \{1,2,\ldots,m\}$, for $1\le j\le s$,
where
\begin{equation}\label{AKeq:def_Ci_red}
    \widetilde{c}_{i,r}^{(j)} = \begin{cases}
        c_{i,r}^{(j)} \quad &\text{ if } i \in \{ 1,\dots, m - \min{(m, w_j)}\},\\
        0 &\text{ if } i \in \{ m - \min{(m, w_j)} + 1,\dots,m \}.
    \end{cases}
\end{equation}

That is, the first $m - \min{(m,w_j)}$ rows of $\akCT_j^{(m)}$ are the same as the rows of the matrix $C_j^{(m)}$, and we set the last $\min{(m,w_j)}$ rows to zero, i.e, if $w_j < m $,
\[ 
\akCT_j^{(m)} = \begin{pmatrix}
c_{1,1}^{(j)}  & c_{1,2}^{(j)}  & \dots &  c_{1,m}^{(j)} & \\
\vdots & \vdots & \ddots &   \vdots   &  \\
c_{(m-w_{j}),1}^{(j)}     & c_{(m-w_{j}),2}^{(j)}     &    \dots    & c_{(m-w_{j}),m}^{(j)} &  \\
0 & 0 & \dots & 0 & \\
\vdots & \vdots & \ddots &  \vdots  &   \\ 
0 & 0 & \dots & 0 & \\
\end{pmatrix}.
\]

We are interested in estimating the quality parameter of the digital net generated by the $\akCT_j^{(m)}$, and state the following result.
\begin{theorem} \label{AKth:main_result_row_red}
    Let $\calP$ be a digital $(t,m,s)$-net over $\F_b$ with generating matrices $C_1^{(m)}, \dots ,C_s^{(m)}$.
Let $\akCT_1^{(m)},\ldots,\akCT_s^{(m)}$ be row reduced with respect to reduction indices $0= w_1 \leq \dots \leq w_s$, and let $\widetilde{t}$ 
be the minimal quality parameter of the net generated by the $\akCT_j^{(m)}$. Then,
\begin{equation}\label{AKeq:rho_row_CT}
    \max\{0,m - \max\{t,w_s\}\} \leq \rho_m \left(\akCT_1^{(m)},\dots,\akCT_s^{(m)}\right) \leq \max\{0,m - w_s \},
\end{equation}
and $\widetilde{t} \leq \min \{m, \max \{t,w_s \} \} $.
\end{theorem}

We have a similar result for column reduced digital nets in \cite[Theorem 1]{AK24}, where we additionally require the $(t,m,s)$-net to be derived from a $(t,s)$-sequence. The main difference is that for the column reduced digital nets, we have $\widetilde{t} \leq \min \{m, t+w_s\}$.

\begin{proof}
If $w_s\ge m$, then we trivially have 
$\rho_m \left(\akCT_1^{(m)},\dots,\akCT_s^{(m)}\right)=0$, as $\akCT_s^{(m)}$ only contains zeros, and \eqref{AKeq:rho_row_CT} holds.

Therefore, we will assume for the rest of the proof that $w_s < m$. 
    If $w_s \leq t$, then 
    \[
    \rho_m \left(\akCT_1^{(m)},\dots,\akCT_s^{(m)}\right) = \rho_m \left(C_1^{(m)},\dots,C_s^{(m)}\right) \geq m-t,
    \]
    since we only consider the linear independence of at most the first $m-t$ rows of any matrix $\akCT_j^{(m)}$.
    
    For the lower bound when $w_s>t$, consider arbitrary integers $d_1,\dots,d_s \geq 0$ with $d_1 + \dots + d_s =  m- w_s$. Then, the collection of the first $d_j$ rows of $\akCT_j^{(m)}$, $1\le j \le s$,  which is the same as the collection of the first $d_j$ rows of $C_j$, $1\le j\le s$, is linearly independent. That is,
\[ 
    \bsc_{1}^{(1)},\dots,\bsc_{d_1}^{(1)}, \bsc_{1}^{(2)},\dots,\bsc_{d_2}^{(2)}, \dots,  \dots,\bsc_{1}^{(s)},\dots,\bsc_{d_s}^{(s)}
\]
are linearly independent. Therefore, 
\[
    \rho_m \left(\akCT_1^{(m)},\dots,\akCT_s^{(m)}\right)  \geq m-w_s.
\]
This concludes the proof of the lower bound. Regarding the upper bound, we use that $w_s < m$, and looking at the rank of $\akCT_s^{(m)}$, we obtain that 
\[
    \rho_m \left(\akCT_1^{(m)},\dots,\akCT_s^{(m)}\right) \leq m-w_s.
\]
Using \eqref{AKeq:rho_t_relation} and the lower bound in \eqref{AKeq:rho_row_CT}, we obtain the upper bound $\widetilde{t} \leq \min \{m, \max \{t,w_s \} \} $ for the minimal quality parameter of the row reduced net. 
\end{proof}

\subsection{The advantage of using reduced digital nets}
 We briefly recapitulate the advantage we gain by using row or column reduced digital nets. Let $N = b^m$ and $\widetilde{\calP} =  \{\bsx_0,\dots,\bsx_{N-1} \}$ be the row or column reduced digital net generated by row or column reduced matrices $\akCT_j^{(m)}$. Let 
 \[
  X  = \begin{pmatrix}
    \bsx_0^\top \\
    \vdots \\
    \bsx_{N-1}^\top \\
\end{pmatrix}  \in \RR^{N \times s}.
 \]
Let $\bsxi_j$ denote the $j$-th column of $X$, i.e., $X=[\bsxi_1,\bsxi_2,\ldots,\bsxi_s]$. Let $A=[\bsa_1,\ldots,\bsa_s ]^\top$, where 
$\bsa_j \in \RR^{1 \times \tau}$ is the $j$-th row of $A$. Since $\bsxi_j$ is computed using the $j$-th reduced generating matrix $\akCT_j^{(m)}$, we can use the repetitiveness that occurs in $\bsxi_j$ due to the zero rows or zero columns of $\akCT_j^{(m)}$. Therefore, to make use of the repetitiveness in $\bsxi_j$, we decompose the product $XA$ as a sum of matrices using the column row expansion as follows,
\begin{equation}\label{AKeq:MM_product}
 XA=[\bsxi_1,\bsxi_2,\ldots,\bsxi_s] \cdot [\bsa_1,\ldots,\bsa_s ]^\top=
 \bsxi_1 \bsa_1 + \bsxi_2 \bsa_2 + \cdots +\bsxi_s \bsa_s.
\end{equation}

In the following remarks, we briefly outline the repetitive pattern in each row reduced and column reduced point set and the computational advantage we gain by using \cite[Algorithm 4]{DEHKL23} and \cite[Algorithm 1]{AK24}, respectively.

\begin{remark}[Row reduced nets]
    Let $\akCT_j^{(m)}$ be the row reduced generating matrix with the last $w_j$ rows set to zero, for $j=1,\dots,s$. Let $s^* \leq s$ be the largest index such that $w_{s^*} <m$; we observe that, for $j>s^*$, $\akCT_j^{(m)}$ is the zero matrix. The $j$-th coordinate $x_{k,j}$ of $\bsx_k \in \widetilde{\calP}$ is obtained by computing 
    \[ x_{k,j} = (\akCT_j^{(m)} \overrightarrow{k}) \cdot (b^{-1},b^{-2},\ldots,b^{-m}), \quad k \in \{0,\dots,N-1\}. \]
Since the last $w_j$ rows of $\akCT_j^{(m)} \overrightarrow{k}$ are zero rows, we observe that all the values of $x_{k,j}$ 
for $k\in\{0,\ldots,N-1\}$ are in the set 
\begin{equation}\label{AKeq:rep_val_rowred}
    \{ 0, 1/b^{m- \min (w_j,m)},2/b^{m- \min (w_j,m)}, \dots, 1- 1/b^{m- \min (w_j,m)} \}.
\end{equation}
Algorithm 4 in \cite{DEHKL23} uses the observation in \eqref{AKeq:rep_val_rowred} to precompute 
the following row vectors for $j = 1,\dots,s$, namely
\[
\bsc_k := \frac{k}{b^{m-w_j}}\bsa_j, \quad k = 0,1,\dots,b^{m-w_j}-1.
\]
Then we can compute 
\[
\bsxi_j \bsa_j = \begin{pmatrix}
    x_{0,j} \\
    \vdots \\
    x_{N-1,j} \\
\end{pmatrix} \bsa_j = \begin{pmatrix}
    \bsc_{\lfloor x_{0,j}b^{m-w_j} \rfloor} \\
    \vdots \\
    \bsc_{\lfloor x_{N-1,j}b^{m-w_j} \rfloor} \\
\end{pmatrix} \in \RR^{N \times \tau},
\]
where we can use the precomputed values.

By this principle, the computational cost for generating points of a row reduced net can be lowered to $$\calO\big( \sum_{j=1}^{s^*} b^m m(m-w_j) \big)$$ as compared to $\calO(s b^m m^2)$ for a non-reduced net. Furthermore, the cost of the matrix-matrix multiplication can be reduced from $\calO(sb^m \tau)$ for the standard non-reduced case to $\calO(\sum_{j=1}^{s^*} b^{m-w_j} \tau)$ for the row reduced case using \cite[Algorithm 4]{DEHKL23}.

\end{remark}

\begin{remark}[Column reduced nets]
    Let $C_1^{(m)},\dots,C_s^{(m)}$ be generating matrices of a $(t,m,s)$-net derived from a digital $(t,s)$-sequence over $\F_b$ and let $\akCT_j^{(m)}$ be the corresponding column reduced generating matrix with the last $w_j$ columns set to zero, for $j=1,\dots,s$. Let $s^* \leq s$ be the largest index such that $w_{s^*} <m$;  we observe that, for $j>s^*$, $\akCT_j^{(m)}$ is the zero matrix. As stated before, the $j$-th coordinate $x_{k,j}$ of $\bsx_k \in \widetilde{\calP}$ 
    is given by 
    \[ x_{k,j} = (\akCT_j^{(m)} \overrightarrow{k}) \cdot (b^{-1},b^{-2},\ldots,b^{-m}), \quad k = 0,\dots,N-1. \]
Since the last $w_j$ columns of $\akCT_j^{(m)}$ are zeros, the last $w_j$ digits of the b-adic expansion of $k$ given by $\overrightarrow{k}$ are insignificant. Thus we observe that the $j$-th column vector of $X$ is given by
\[ 
\bsxi_j = (\underbrace{X_j,\ldots,X_j}_{b^{w_j}\ \rm times})^\top,
\]
where
\[
 X_j= \left(\left(\akCT_j^{(m)} \overrightarrow{0}  \right)\cdot(b^{-1},\ldots,b^{-m}),\ldots,
             \left(\akCT_j^{(m)} \overrightarrow{(b^{m-w_j}-1)}  \right)\cdot(b^{-1},\ldots,b^{-m})\right)^\top.
\]
Thus, to compute $XA$, we first compute $X_j \bsa_j \in \RR^{b^{m-w_j \times \tau}}$ instead of $\bsxi_j \bsa_j$, and we duplicate and enlarge it to match the dimension of $X_{j-1} \bsa_{j-1}$ and add this recursively to compute $XA$. For more details, see \cite[Algorithm 1]{AK24}. 

Using column reduced digital nets, we can obtain a speed-up in both the point generation step and the matrix-matrix multiplication. In \cite[Algorithm 1]{AK24}, the point generation step requires $\calO(\sum_{j=1}^{s^*} b^{m-w_j} m (m-w_j))$ operations and the matrix-matrix multiplication requires $\calO(\sum_{j=1}^{s^*} b^{m-w_j} \tau)$. Since the $j$-th column vector $\bsxi_j$ has a repetitive structure, we obtain a further reduction in the number of additions in \eqref{AKeq:MM_product}, which we do not have in the row reduced case. However, we only use nets derived from $(t,s)$-sequences for the column reduced digital nets, in order to still be able to give bounds on the quality parameter $\widetilde{t}$ of the reduced net. We do not need to impose this restriction for row reduced digital nets. 

\end{remark}

\subsection{Projections of digital nets}
The quality parameter $t$ determines the equidistribution properties of a $(t,m,s)$-net and hence also influences the error of a QMC rule that uses the points of the net. Due to the importance of the quality parameter, we also consider a finer measure, which is to look at the quality parameters of the projections of a $(t,m,s)$-net. 

Let $[s]:=\{1,\ldots,s\}$. For any  $\setu\neq\emptyset$, $\setu\subseteq [s]$, the projection of a $(t,m,s)$-net onto those components with indices in $\setu$ can be considered as a $(t_\setu,m,\AKabs{\setu})$-net for a suitable choice of $t_\setu \in \{1,\ldots,m\}$. Thus, we can view a $(t,m,s)$-net also as a $((t_\setu)_{\emptyset\neq\setu\subseteq [s]},m,s)$-net.

The notion of a $((t_\setu)_{\setu \subseteq [s]},s)$-sequence is defined analogously. Moreover, for $\setu\neq \emptyset$, we write 
$\overline{\setu}:=\max (\setu)$.
We have the following result for projections of row reduced nets.

\begin{corollary}\label{AKlem:t_value_red_proj}
Let $\calP$ be a digital $((t_\setu)_{\setu \subseteq [s]},m,s)$-net over $\F_b$ with generating matrices $C_1^{(m)}, \dots ,C_s^{(m)}$,
where we assume that $m\ge t$. 
Let $\akCT_1^{(m)},\ldots,\akCT_s^{(m)}$ be the row reduced generating matrices with respect to reduction indices $0= w_1 \leq \dots \leq w_s$ and let $(\widetilde{t}_\setu)_{\setu \subseteq [s]}$ 
be the minimal quality parameters of the projections of the net generated by the $\akCT_j^{(m)}$. Then, for every non-empty $\setu\subseteq [s]$,
\begin{equation*}
    \max\{0,m- \max \{w_{\overline{\setu}},t_{\setu}\}\} \leq \rho_m((\akCT_j^{(m)})_{j\in\setu}) \leq \max\{0,m - w_{\overline{\setu}}\},
\end{equation*}
and $ \widetilde{t}_{\setu} \le \min\{m,\max \{w_{\overline{\setu}} ,t_{\setu}\}\}$.

\end{corollary}

We have a corresponding result in \cite[Corollary 1]{AK24} for column reduced digital nets derived from $(t,s)$-sequences, where the quality parameter of the column reduced net is given by $ \widetilde{t}_{\setu} \le \min\{m,w_{\overline{\setu}} + t_{\setu}\}$.

\subsection{Error analysis}

The main motivation for this paper is to approximate integral \eqref{AKeq:integral} using QMC rule \eqref{AKeq:QMC_approx}. While we gain an advantage in the computational cost by using row-reduced digital nets as opposed to non-reduced nets, there is a certain loss in the quality parameter of reduced nets as formalized in \cref{AKth:main_result_row_red}.

In this section, we focus on the integration error obtained by using a QMC rule. The error analysis here is similar to the one in \cite{AK24}, where we gave an upper bound on the quality parameter of column reduced nets. In the present paper, we have given a corresponding result for row reduced digital nets in \cref{AKth:main_result_row_red} above. There is another error estimate using row reduced digital nets given in \cite[Theorem 4]{DEHKL23}, but there the authors do not have an estimate for the quality parameter of the reduced net, and therefore use a different method to obtain an error bound. Here, we use the result in \cref{AKlem:t_value_red_proj} for the error analysis.

We start by considering 
the \emph{weighted star discrepancy}, which is (via the well-known \emph{Koksma-Hlawka inequality} or its weighted version, see, e.g., \cite{DKP22,DP10,N92}), a measure of the worst-case quadrature error for a QMC rule with node set $Q$, with $b^m$ nodes, defined as
\begin{equation}
    D_{b^m,\bsgamma}^{*}(Q) := \sup_{x \in (0,1]^s} \max_{\emptyset \neq \fraku \subseteq [s]} \gamma_{\fraku} \AKabs{\Delta_{Q,\fraku}(\bsx)}, 
\end{equation}
where
\begin{equation}\label{AKdef:DeltaP}
    \Delta_{Q,\fraku}(\bsx) := \frac{ \# \{(y_1,\ldots,y_s)\in Q  \colon y_j < x_j, \, \forall j \in \fraku \}}{b^m}- \prod_{j\in \fraku} x_j.
\end{equation}
Here, the \textit{weights} $\gamma_{\fraku}$ are non-negative real numbers, which model the influence of the variables with indices in $\fraku$. The concept of weights was first introduced in the seminal paper \cite{SW98}  by Sloan and Wo\'{z}niakowski, and is nowadays a standard way of modeling high-dimensional problems in the context of QMC methods. Here, we assume that the weights are so-called (positive) \textit{product weights}. That is, we assume a non-increasing sequence $(\gamma_j)_{j\ge 1}$ of positive reals, and put 
$\gamma_{\fraku}:=\prod_{j\in\fraku} \gamma_j$, where $\gamma_{\emptyset}:=1$.
We refer to \cite{DP10} for further details regarding weights in the context of $(t,m,s)$-nets.

Using Lemma 1 in \cite{P18}, we have,
\begin{eqnarray*}
    D_{b^m,\bsgamma}^{*}(Q) &=&  \max_{\emptyset \neq \fraku \subseteq [s]} 
    \sup_{\bsx \in (0,1]^s} \gamma_{\fraku} \AKabs{\Delta_{Q,\fraku}(\bsx)}\\
    &=&  \max_{\emptyset \neq \fraku \subseteq [s]} \gamma_{\fraku}
    \sup_{\bsx \in (0,1]^s}  \AKabs{\Delta_{Q,\fraku}(\bsx)}.
\end{eqnarray*}

In \cite{FK13}, the authors provide discrepancy bounds for projections of $(t,m,s)$-nets. Let $\calP$ be a digital $((t_\setu)_{\setu \subseteq [s]},m,s)$-net over $\F_b$ with $m \times m $ generating matrices $C_1^{(m)}, \dots ,C_s^{(m)}$, 
where $m\ge t$. Let $\widetilde{\calP}$ be the corresponding row reduced digital net based on the reduction indices $0=w_1 \le w_2 \le \cdots \le w_s$, 
and let $(\widetilde{t}_\setu)_{\setu \subseteq [s]}$ be the minimal quality parameters of the projections of $\widetilde{\calP}$. Using a result from \cite{FK13}, which
gives a relation between the discrepancy and the quality paramater of $(t,m,s)$-nets, we obtain, analogously to what is done in \cite[Section 4]{AK24},
\begin{equation}\label{AKeq:disc_bound_global}
 D_{b^m,\bsgamma}^{*}(\widetilde{\calP}) \le  \max\left\{
 \max_{\substack{\emptyset \neq \fraku \subseteq [s]\\ \setu\not\subseteq [s^*]}} \gamma_{\fraku}
 ,\max_{\substack{\fraku \subseteq [s^*]\\ \AKabs{\setu}=1}} \gamma_{\setu}\,
 \frac{b^{\widetilde{t}_\setu}}{b^m}
 ,\max_{\substack{\fraku \subseteq [s^*]\\ \AKabs{\setu}\ge 2}} \gamma_{\setu}\,
 \frac{b^{\widetilde{t}_\setu}}{b^m}\sum_{v=0}^{\AKabs{\setu}-1}   a_{v,b}^{(\AKabs{\setu})} m^v 
 \right\},
\end{equation}
where
\begin{eqnarray*}
a_{v,b}^{(\AKabs{\setu})}&=&{\AKabs{\setu}-2 \choose v} \left(\frac{b+2}{2}\right)^{\AKabs{\setu}-2-v}\frac{(b-1)^{v}}{2^{v}v!} (a_{0,b}^{(2)}+ \AKabs{\setu}^2-4)\nonumber\\
&&+{\AKabs{\setu}-2 \choose v-1} \left(\frac{b+2}{2}\right)^{\AKabs{\setu}-1-v}\frac{(b-1)^{v-1}}{2^{v-1}v!} a_{1,b}^{(2)},
\end{eqnarray*}
for $0\le v\le \AKabs{\setu}-1$,
with
\[
a_{0,b}^{(2)}=\begin{cases}
                                     \frac{b+8}{4}&\mbox{if $b$ is even,}\\ \\
                                     \frac{b+4}{2}&\mbox{if $b$ is odd,}
                                     \end{cases}
\hspace{1cm} \mbox{and} \hspace{1cm} 
a_{1,b}^{(2)}=\begin{cases}
                                     \frac{b^2}{4(b+1)}&\mbox{if $b$ is even,}\\ \\
                                     \frac{b-1}{4}&\mbox{if $b$ is odd.}
                                     \end{cases}
\]

To analyze the first term in \eqref{AKeq:disc_bound_global}, we proceed as in \cite{DEHKL23}, that is, we use that $w_j\ge m$ if $j\in \setu\setminus [s^\ast]$, and obtain for $\setv = \setu \cap [s^\ast]$ that
\begin{equation}\label{AKeq:disc_term_1}
 \gamma_{\setu} \le \gamma_{\setv} \gamma_{\setu\setminus\setv} \frac{1}{b^m}\prod_{j\in\setu\setminus\setv} (1+b^{w_j})
 \le \frac{1}{b^m}  \prod_{j\in \setu} \gamma_j (1 + b^{w_j}). 
\end{equation}

For the second term in \eqref{AKeq:disc_bound_global}, note that we have 
$\setu=\{j\}$ for some $j\in [s^*]$, and hence using \cref{AKlem:t_value_red_proj}, we have
$\widetilde{t}_{\setu} \le \min\{m,\max\{ w_j, t_{\{j\}} \} \}$. Consequently,
\begin{equation}\label{AKeq:disc_term_2}
 \max_{\substack{\fraku \subseteq [s^*]\\ \AKabs{\setu}=1}} \gamma_{\setu}
 \frac{b^{\widetilde{t}_\setu}}{b^m} \le 
 \max_{j\in [s^*]} \gamma_j \frac{b^{\min\{m,\max\{ w_j, t_{\{j\}} \}\}}}{b^m}.
\end{equation}

For the third term in \eqref{AKeq:disc_bound_global}, we again use \cref{AKlem:t_value_red_proj}, and obtain 
\begin{equation}\label{AKeq:disc_term_3}
 \max_{\substack{\fraku \subseteq [s^*]\\ \AKabs{\setu}\ge 2}} \gamma_{\setu}\,
 \frac{b^{\widetilde{t}_\setu}}{b^m}\sum_{v=0}^{\AKabs{\setu}-1}   a_{v,b}^{(\AKabs{\setu})} m^v
 \le 
 \max_{\substack{\fraku \subseteq [s^*]\\ \AKabs{\setu}\ge 2}} \gamma_{\setu}\,
 \frac{b^{\min \{m, \max \{ w_{\overline{\setu}}, t_{\setu} \} \}}}{b^m}\sum_{v=0}^{\AKabs{\setu}-1}   a_{v,b}^{(\AKabs{\setu})} m^v.
\end{equation}
Using these estimates in \eqref{AKeq:disc_bound_global}, we obtain
\begin{equation}\label{AKeq:disc_bound_global_2}
\begin{aligned}
    D_{b^m,\bsgamma}^{*}(\widetilde{\calP}) \le \max  \Biggl\{
  \max_{\substack{\emptyset \neq \fraku \subseteq [s]\\ \setu\not\subseteq [s^*]}}
  \frac{1}{b^m}  \prod_{j\in \setu} \gamma_j (1 + b^{w_j}),
 \max_{j\in [s^*]} \gamma_j \frac{b^{\min \{ m, \max \{ w_j, t_{\{j\}} \} \} }}{b^m} &, \\  \max_{\substack{\fraku \subseteq [s^*]\\ \AKabs{\setu}\ge 2}} \gamma_{\setu}\,
 \frac{b^{\min \{m, \max \{ w_{\overline{\setu}},t_{\setu} \}\}}}{b^m}\sum_{v=0}^{\AKabs{\setu}-1} a_{v,b}^{(\AKabs{\setu})} m^v \Biggr\}. & \\
\end{aligned}
\end{equation}

The above estimate is similar to the error bound in \cite{AK24} for column reduced digital nets, since the estimate for the quality parameter of row reduced digital nets as given in \cref{AKlem:t_value_red_proj}, is similar to the corresponding estimate for column reduced digital nets. Analogously to \cite[Remark 5]{AK24}, the first two terms of \eqref{AKeq:disc_bound_global_2} can be made independent of the dimension $s$ by choosing reduction indices $w_j$ that increase sufficiently fast. The third term depends on the interplay between the weights $\gamma$ and the quality parameters $t_{\setu}$ of the projections of $\calP$, where for suitable choice of reduction indices $w_j$, small quality parameters for the projections and sufficiently fast decaying weights, we can obtain tighter error bounds. For more details, we refer to \cite[Remark 5]{AK24}. 

As mentioned before, in \cite{DEHKL23}, the authors also provide error bounds using row reduced digital nets where they assume that the quality parameter of the row reduced net $\widetilde{t}_\setu$ is known, which may frequently not be the case in practice. However, since we can estimate $\widetilde{t}_\setu$ using \cref{AKlem:t_value_red_proj}, our contribution is to provide explicit error bounds using the quality parameter $t_\setu$ of the original net and the reduction indices $w_j$.

\section{Combination of column and row reduced digital nets} \label{AKsubsec:combi_col_row}
In \cite{AK24} and \cite{DEHKL23}, the authors studied column reduced digital nets and row reduced digital nets, respectively. It is a natural question to ask what happens when we combine the two and look at column-row reduced digital nets. That is, for each generating matrix $C_j^{(m)}$, $1\le j\le s$, we set the last $w_j$ columns and the last $w_j$ rows to zero. Since we set some columns to zero, we impose the condition that the generating matrices for the $(t,m,s)$-nets have to be derived from $(t,s)$-sequences, as it was also done in \cite{AK24}.

Let $0=w_1 \le \cdots  \le w_s \in \N_0$ be the reduction indices for the generating matrices $C_1^{(m)},\ldots,C_s^{(m)}$, where the $C_j^{(m)} = (c_{i,r}^{(j)})$, $i,r\in \{1,2,\ldots,m\}$, are derived from a digital $(t,s)$-sequence over $\F_b$. We derive the corresponding column-row reduced matrices $\akCT_1^{(m)},\ldots,\akCT_s^{(m)}$, with $\akCT_j^{(m)}=(\widetilde{c}_{i,r}^{(j)})$, $i,r\in \{1,2,\ldots,m\}$, for $1\le j\le s$,
where
\begin{equation}\label{AKeq:def_Ci_cr_red}
    \widetilde{c}_{i,r}^{(j)} = \begin{cases}
        c_{i,r}^{(j)} \quad &\text{ if } i\mbox{ and }r \in \{ 1,\dots, m - \min{(m, w_j)}\},\\
        0 &\text{otherwise.}
    \end{cases}
\end{equation}

We observe that \cite[Theorem 1]{AK24} also holds for column-row reduced digital nets. Since the linear independence parameter $\rho_m(\akCT_1^{(m)}, \dots,\akCT_s^{(m)} )$ depends on systems of row vectors of the $\akCT_j^{(m)}$ being linearly independent (see \cref{AKdef:rho_m}), this measure is more sensitive to setting the last few columns to zero as compared to setting the last few rows to zero. In particular, since the generating matrices are derived from a $(t,s)$-sequence, we have that for the left upper $(m-w_s) \times (m-w_s)$ submatrices of $C_1,\dots,C_s$, denoted by $C_j^{(m-w_s)}$,
\[
\rho_{(m-w_s)}(C_1^{(m-w_s)},\dots,C_s^{(m-w_s)}) \geq m-w_s-t.
\]
This lower bound still holds for column-row reduced nets with the same argument as presented in \cite[Theorem 1]{AK24}.
 
We recall \cite[Theorem 1]{AK24} here with the statement modified for column-row reduced nets. 

\begin{lemma}
    Let $\calP$ be a digital $(t,m,s)$-net over $\F_b$ with generating matrices 
$C_1^{(m)}, \dots ,C_s^{(m)}$ derived from a digital $(t,s)$-sequence over $\F_b$, 
where we assume that $m\ge t$. 
Let $\akCT_1^{(m)},\ldots,\akCT_s^{(m)}$ be as defined in \eqref{AKeq:def_Ci_cr_red} with respect to reduction indices $0= w_1 \leq \dots \leq w_s$ and let $\widetilde{t}$ 
be the minimal quality parameter of the net generated by the $\akCT_j^{(m)}$. Then,
\begin{equation}\label{AKeq:claim_rho_CT}
    \max\{0,m-w_s-t\} \leq \rho_m \left(\akCT_1^{(m)},\dots,\akCT_s^{(m)}\right) \leq \max\{0,m - w_s\},
\end{equation}
and $ \widetilde{t} \le \min\{m,w_s +t\}$.

\medskip

Furthermore, if $\calP$ is a strict digital $(t,m,s)$-net, it is true that
\begin{equation}\label{AKeq:claim_rho_CT_strict}
    \rho_m \left(\akCT_1^{(m)},\dots,\akCT_s^{(m)}\right) \leq \max\{0,m - \max\{t,w_s\}\}.
\end{equation}
\end{lemma}
\begin{proof}
    See  \cite[Proof of Theorem 1]{AK24}.
\end{proof}

\begin{remark} \label{AKrem:col_row_runtime}
    We can modify \cite[Algorithm 1]{AK24} to obtain a slightly faster algorithm for matrix-matrix multiplication for column-row reduced digital nets. In particular, we can obtain a slight speed-up in the computation of the points of the reduced digital nets, where we can resize $\akCT_j^{(m)}$ to $\akCT_j^{(m-w_j)}$. That is, the theoretical runtime of the column-row reduced algorithm is 
\[
\calO \left( \sum_{j=1}^{s^*} b^{m-w_j} (\tau + (m-w_j)^2 ) \right).
\]
In comparison to the column reduced matrix-matrix product algorithm, we see that there is an improvement in the point generation step.
\end{remark}

\begin{remark}
So far, for column-row reduced digital nets, we always set an equal number of rows and columns to zero for each generating matrix $C_j^{(m)}, 1\leq j\leq s $, where the $C_j^{(m)}$ are derived from a digital $(t,s)$-sequence over $\F_b$. Theoretically, we could also consider the scenario where we set a different number of rows and columns to zero. In this case, we can obtain the following generalization of \cite[Theorem 1]{AK24} and \cref{AKth:main_result_row_red}. Let $0= w_1^r \leq \dots \leq w_s^r$ be the reduction indices for the rows and $0= w_1^c \leq \dots \leq w_s^c$ be the reduction indices for the columns of $C_1^{(m)}, \dots, C_s^{(m)}$. Let $\akCT_1^{(m)},\ldots,\akCT_s^{(m)}$ be the corresponding column-row reduced generating matrices, and let $\widetilde{t}$ 
be the minimal quality parameter of the net generated by the $\akCT_j^{(m)}$. Then, 
    \[
    \max\{ 0, m- \max \{ w_s^c + t, w_s^r \} \} \leq \rho_m(\akCT_1,\dots,\akCT_s) \leq \max\{ 0, m- \max \{w_s^c, w_s^r \} \},
    \]
and $\widetilde{t} \leq \min\{m, \max \{ w_s^c + t, w_s^r \} \}$.

\end{remark}

\section{Numerical experiments}

\begin{figure}[t]
	\begin{subfigure}[t]{0.48\textwidth}
		\includegraphics[width=\textwidth]{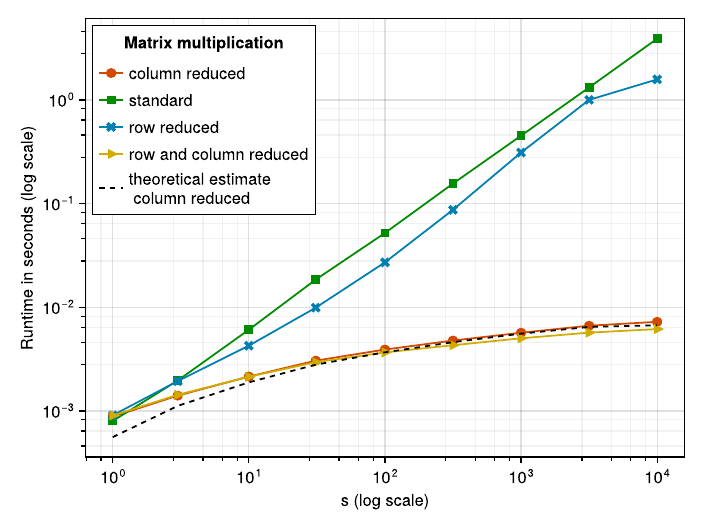}
		\subcaption{$w_j = \min(\lfloor \log_2(j) \rfloor, m)$}
		\label{AKfig:wt_inc_1}
	\end{subfigure}
	\hfill
	\begin{subfigure}[t]{0.48\textwidth}
		\includegraphics[width=\textwidth]{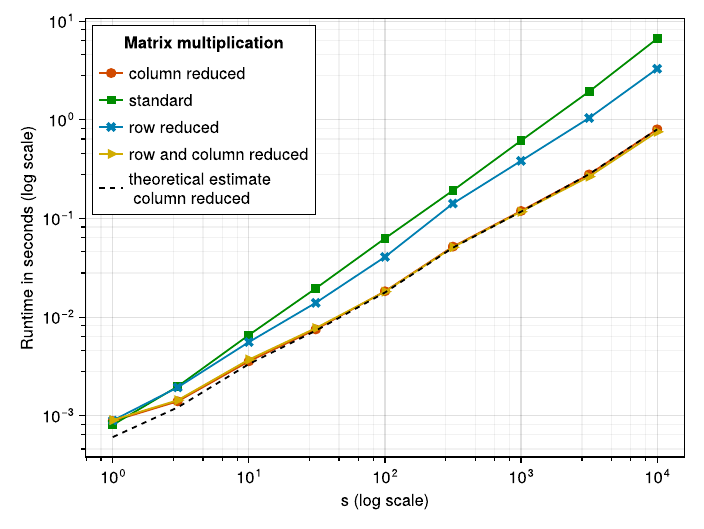}
		\subcaption{$ w_j = \min(\lfloor \log_2(j^{1/2}) \rfloor, m)$}
        \label{AKfig:wt_inc_2}
	\end{subfigure}
\caption{$m=12,\tau = 20$, varying $w_j$} \label{AKfig:weight_comp}
\end{figure}

\begin{figure}[t]
    \centering
		\includegraphics[width=0.48\textwidth]{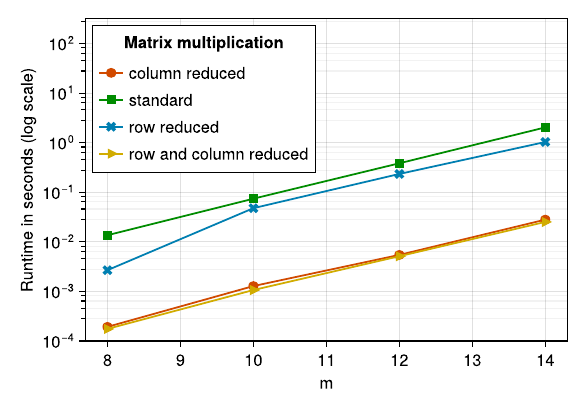}
	\caption{$s=800, \tau = 20$, varying $m$} \label{AKfig:vary_m}
\end{figure}

In \cite[Section 5]{AK24}, we tested the runtime of \cite[Algorithm 1]{AK24} which computes the matrix-matrix product $XA$, where $X$ is generated using column reduced digital nets and $A$ is an arbitrary $s \times \tau$ matrix. We compared this runtime to the runtime of the standard matrix multiplication and the runtime of \cite[Algorithm 4]{DEHKL23}, which uses row reduced digital nets.

In this section, we revisit the plots from \cite[Section 5]{AK24}, where we additionally compare the computational performance of column-row reduced digital nets. In particular, we modify \cite[Algorithm 1]{AK24} as described in \cref{AKrem:col_row_runtime}.

All the above-mentioned algorithms are implemented in the Julia programming language (Version 1.9.3).\footnote{Source code available at \url{https://github.com/Vishnupriya-Anupindi/ReducedDigitalNets.jl}}
We used random binary matrices as the generating matrices $C_1^{(m)},\ldots,C_s^{(m)} \in \mathbb{F}_2^{m \times m}$, since the matrix-matrix product computation does not change if we use random matrices or generating matrices of, e.g., Sobol' or Niederreiter sequences. 

\cref{AKfig:weight_comp} shows the runtime evolution as we increase the dimension $s$ but keep the other parameters fixed at $b = 2, m = 12$, and $\tau = 20$. In \cref{AKfig:wt_inc_1}, the reduction indices are given by $w_j = \min(\lfloor \log_2(j) \rfloor, m)$. Here we see that the runtime of the matrix-matrix product using column-row reduced digital nets (labelled as row and column reduced) and column reduced digital nets are almost the same as expected and both of these perform better than using row reduced digital nets in \cite[Algorithm 4]{DEHKL23} and also the standard multiplication. However, when the reduction indices increase more slowly, for example, $ w_j = \min(\lfloor \log_2(j^{1/2}) \rfloor, m)$ as shown in \cref{AKfig:wt_inc_2}, the difference in the performance between the different algorithms reduces. 

\cref{AKfig:vary_m} compares the runtime evolution as we increase the size of the generating matrices $m$ with fixed values $b = 2, s = 800$, and $\tau = 20$. Also here, the runtime of the matrix-matrix product using column-row reduced (labelled as row and column reduced) and column reduced matrices are almost the same as expected.

More extensive numerical tests are left open for future research.

\section*{Acknowledgements}

The authors acknowledge the support of the Austrian Science Fund (FWF) Project P34808. For open access purposes, the authors have applied a CC BY public copyright license to any author accepted manuscript version arising from this submission.

%
%

\end{document}

%% file: macros.tex
\newcommand{\bsa}{{\boldsymbol{a}}}

\newcommand{\bsc}{{\boldsymbol{c}}}

\newcommand{\bsx}{{\boldsymbol{x}}}

\newcommand{\bsgamma}{{\boldsymbol{\gamma}}}

\newcommand{\bsxi}{{\boldsymbol{\xi}}}

\newcommand{\rd}{{\mathrm{d}}}

\newcommand{\F}{{\mathbb{F}}} 
\newcommand{\N}{{\mathbb{N}}} 
\newcommand{\RR}{{\mathbb{R}}} 

\DeclareSymbolFont{bbold}{U}{bbold}{m}{n}
\DeclareSymbolFontAlphabet{\mathbbold}{bbold}

\newcommand{\calO}{{\mathcal{O}}}
\newcommand{\calP}{{\mathcal{P}}}

\newcommand{\fraku}{{\mathfrak{u}}}

\newcommand{\setu}{{\mathfrak{u}}}
\newcommand{\setv}{{\mathfrak{v}}}

